\documentclass[11pt,letterpaper]{article} 

\pdfoutput=1

\usepackage{graphics,epsfig,graphicx,color,xcolor}
\graphicspath{{../Figures/}}
\usepackage[caption = false]{subfig}
\usepackage{float}
\usepackage{capt-of}

\usepackage{amssymb,latexsym}

\usepackage{setspace}

\usepackage{amsmath}

\usepackage{mathrsfs,amsfonts}

\usepackage{amsthm,amsxtra}

\usepackage[final]{showkeys}

\usepackage{stmaryrd}

\usepackage{algorithm}
\usepackage{algorithmicx}
\usepackage{algpseudocode}

\usepackage[openbib]{currvita}
\usepackage{bibentry}

\usepackage{etaremune}
\usepackage{enumitem} 

\newif\ifPDF
\ifx\pdfoutput\undefined
	\PDFfalse
\else
	\ifnum\pdfoutput > 0
		\PDFtrue
	\else
		\PDFfalse
	\fi
\fi

\ifPDF
	\usepackage{pdftricks}
	\begin{psinputs}
		\usepackage{pstricks}
		\usepackage{pstcol}
		\usepackage{pst-plot}
		\usepackage{pst-tree}
		\usepackage{pst-eps}
		\usepackage{multido}
		\usepackage{pst-node}
		\usepackage{pst-eps}
	\end{psinputs}
\else
	\usepackage{pstricks}
\fi

\ifPDF
	\usepackage[debug,pdftex,colorlinks=true, 
	linkcolor=blue, bookmarksopen=false,
	plainpages=false,pdfpagelabels]{hyperref}
\else
	\usepackage[dvips]{hyperref}
\fi

\usepackage{fancyhdr}

\usepackage{comment}

\usepackage[toc,page]{appendix}

\usepackage{cancel}

\usepackage{multirow}
\usepackage{tabularx}
 
\newtheorem{theorem}{Theorem}[section]
\newtheorem{lemma}[theorem]{Lemma}

\newtheorem{proposition}[theorem]{Proposition} 
\newtheorem{remark}[theorem]{Remark}

\newcommand{\dint}{\displaystyle\int}

\newcommand{\eps}{\varepsilon}

\newcommand{\fc}{\mathfrak{c}}

\newcommand{\bbR}{\mathbb R} \newcommand{\bbS}{\mathbb S}

\newcommand{\bnu}{{\boldsymbol \nu}}

 \newcommand{\bv}{\mathbf v} 
 \newcommand{\bx}{\mathbf x}

\newcommand{\cA}{\mathcal A} 
\newcommand{\cC}{\mathcal C}

 \newcommand{\cL}{\mathcal L}
 
\newcommand{\cO}{\mathcal O}

\newcommand{\aver}[1]{\langle {#1} \rangle}

\usepackage[normalem]{ulem}
\usepackage{xcolor} 
 \newcommand{\redsout}{\bgroup\markoverwith{\textcolor{red}{\rule[1ex]{3pt}{.8pt}}}\ULon}
\newcommand{\wt}{\widetilde}
\newcommand{\wh}{\widehat}

\newenvironment{keywords}{\noindent{\bf Key words.}\small}{\par\vspace{1ex}}
\newenvironment{AMS}{\noindent{\bf AMS subject classifications 2010.}\small}{\par}

\newcommand{\DELETE}[1]{}

\newcommand{\LC}{\left(}
\newcommand{\RC}{\right)}
\newcommand{\p}{\partial}

\usepackage{todonotes}

\title{Inverse transport and diffusion problems in photoacoustic imaging with nonlinear absorption}

\author{
	Ru-Yu Lai\thanks{School of Mathematics, University of Minnesota, Minneapolis, MN 55455; 	\href{mailto:rylai@umn.edu}{rylai@umn.edu}
	}
	\and
	Kui Ren\thanks{Department of Applied Physics and Applied Mathematics, Columbia University, New York, NY 10027;
		\href{mailto:kr2002@columbia.edu}{kr2002@columbia.edu}
	}
	\and
	Ting Zhou\thanks{Department of Mathematics, Northeastern University, Boston, MA 02115; 
		\href{mailto:t.zhou@northeastern.edu}{t.zhou@northeastern.edu}
	}
}

\begin{document}


\maketitle



\begin{abstract}
Motivated by applications in imaging nonlinear optical absorption by photoacoustic tomography (PAT), we study in this work inverse coefficient problems for a semilinear radiative transport equation and its diffusion approximation with internal data that are functionals of the coefficients and the solutions to the equations. Based on the techniques of first- and second-order linearization, we derive uniqueness and stability results for the inverse problems. For uncertainty quantification purpose, we also establish the stability of the reconstruction of the absorption coefficients with respect to the change in the scattering coefficient.
\end{abstract}






\section{Introduction}
\label{SEC:Intro}

This paper is devoted to the study of inverse coefficient problems in quantitative photoacoustic imaging of optically heterogeneous materials, such as biological tissues, with a nonlinear absorption effect. To describe the problem, let us denote the underlying medium to be probed by $\Omega\subseteq\bbR^d$ ($d\ge 2$), an open bounded convex domain with smooth boundary $\partial\Omega$. We denote by $\bbS^{d-1}$ the unit sphere in $\bbR^d$, and define the phase space $X:=\Omega\times\bbS^{d-1}$ as well as the incoming boundary of the phase space 
$$
\Gamma_-:=\{(\bx, \bv)\ |\  (\bx, \bv)\in\partial\Omega\times\bbS^{d-1}\ \mbox{s.t.}\ -\bnu(\bx)\cdot \bv > 0\},$$ 
where $\bnu(\bx)$ is the unit outer normal vector at $\bx\in\partial\Omega$. In a photoacoustic experiment, we send near infra-red (NIR) photons into the media $\Omega$. The density of the photons at $\bx\in\Omega$ traveling in the direction $\bv\in\bbS^{d-1}$, $u(\bx, \bv)$, solves the following semilinear radiative transport equation~\cite{Arridge-IP99,BaJoJu-IP10,ReZh-SIAM21}
\begin{equation}\label{EQ:ERT TP}
\begin{array}{rcll}
\bv \cdot \nabla u(\bx, \bv) + \sigma_a(\bx) u(\bx, \bv) + \sigma_b \aver{u} u(\bx, \bv) &=& \sigma_s(\bx) K (u)(\bx, \bv), &\text{in}\ X\\
u(\bx, \bv) &= & g(\bx, \bv), &\text{on}\ \Gamma_{-}
\end{array}
\end{equation}
where $\sigma_a, \sigma_b$ are the single-photon and two-photon absorption coefficients, respectively, and $\sigma_s$ is scattering coefficient.
We denote by $\aver{u}$ the integral of $u(\bx, \bv)$ over the variable $\bv$, that is, 
\begin{equation*}\label{EQ:Aver}
	\aver{u}:=\int_{\bbS^{d-1}} u(\bx, \bv) d\bv\,,
\end{equation*}
with $d\bv$ being the \emph{normalized} surface measure on $\bbS^{d-1}$ (that is, $\int_{\bbS^{d-1}} d\bv=1$). The linear scattering operator $K$ is defined through the relation
\begin{equation*}\label{EQ:Scattering Op}
\begin{aligned}
K (u)(\bx, \bv) : =  \int_{\bbS^{d-1}}\Big\{ \Theta(\bv, \bv') u(\bx, \bv')  - \Theta(\bv', \bv) u(\bx, \bv)\Big\} d\bv',
\end{aligned}
\end{equation*}
with the non-negative kernel $\Theta(\bv, \bv')\geq 0$ satisfying the normalization conditions 
\[
\dint_{\bbS^{d-1}} \Theta(\bv, \bv') d\bv' = \dint_{\bbS^{d-1}} \Theta(\bv, \bv') d\bv = 1.
\]

The pressure field generated by the photoacoustic effect can be written as~\cite{FiScSc-PRE07}
\begin{equation}\label{EQ:Data TP}
	H_T(\bx) = \Xi(\bx) \Big[\sigma_a(\bx) \aver{u}(\bx) + \sigma_b(\bx) \aver{u}^2(\bx)\Big], \qquad \bx\in \bar\Omega.
\end{equation}
where $\Xi$ is the Gr\"uneisen coefficient that describes underlying medium's photoacoustic efficiency. This initial pressure field generated by single-photon and two-photon absorption processes evolves, in the form of ultrasound, according to the classical acoustic wave equation~\cite{BaUh-IP10,FiScSc-PRE07}. Through the measurement of the ultrasound data reaching the surface of the medium, one can reconstruct the internal information $H_T(\bx)$. This is by now a well-established process; see, for instance, ~\cite{AmBrJuWa-LNM12,BuMaHaPa-PRE07,CoArBe-IP07,HaScSc-M2AS05,Hristova-IP09,KiSc-SIAM13,KuKu-EJAM08,StUh-IP09} and references therein for more details. 

The objective of this paper is on the second step of the photoacoustic imaging technology: to reconstruct the optical coefficients $\sigma_a, \sigma_b, \sigma_s$ and possibly $\Xi$ from the internal information $H_T$ reconstructed from the acoustic measurement. What makes our study different from existing results on quantitative photoacoustic imaging, for instance those in~\cite{BaRe-CM11,BaUh-IP10,GaOsZh-LNM12,LaCoZhBe-AO10,MaRe-CMS14,PuCoArKaTa-IP14,ReZhZh-IP15,SaTaCoAr-IP13}, is that the transport model~\eqref{EQ:ERT TP} we consider here contains the semilinear term that describes the two-photon absorption effect of the underlying medium~\cite{BaReZh-JBO18,ReZh-SIAM18}. This additional nonlinearity makes the analysis of the inverse problem much more complicated~\cite{ReZh-SIAM21,StZh-arXiv21}.

\paragraph{Diffusion approximation.} When the underlying medium has very strong scattering but weak absorption, one can approximate the transport equation model with a diffusion equation model that is easier to deal with. This is a well-established result in kinetic theory in the absence of the semilinear term $\sigma_b\aver{u} u$ in~\eqref{EQ:ERT TP}; see for instance ~\cite{DaLi-Book93-6} for a detailed mathematical derivation. In the presence of the semilinear term $\sigma_b\aver{u} u$, the diffusion approximation follows straightforwardly from the classical theory under the assumption that the transport solution is at most $\cO(1)$. This is indeed the regime where our study in the rest of the paper will be, that is, for small boundary data. Therefore, we write down the following semilinear diffusion approximation without further justification, and with a little bit abuse of notations:
\begin{align}\label{EQ:Diff TP}
\begin{array}{rcll}
-\nabla \cdot\gamma(\bx) \nabla u(\bx) +\sigma_a(\bx) u(\bx)+\sigma_b(\bx) u(\bx) u(\bx) &=& 0, & \mbox{in}\ \ \Omega\\
u(\bx) &=& g(\bx), &\mbox{on}\ \partial\Omega
\end{array}
\end{align}
where $\gamma$ is related to $\sigma_a$, $\sigma_b$ and $\sigma_s$. The internal data in the diffusion approximation now take the form
\begin{equation}\label{EQ:Data Diff}
	H_D(\bx)=\Xi[\sigma_a(\bx) u(\bx)+\sigma_b(\bx) u(\bx) u(\bx)]\qquad \bx\in \bar\Omega.
\end{equation}
The inverse problem in this case is to reconstruct information on $\Xi, \gamma, \sigma_a$ and $\sigma_b$ from the data in the form of $H_D$. 

Note that the diffusion model we take here has the semilinear term $u(\bx) u(\bx)$ instead of $|u(\bx)|u(\bx)$ as in~\cite{ReZh-SIAM18}. Using the $|u(\bx)|u(\bx)$ term will force the solution to the diffusion equation to be non-negative, a property that is desired for the problem to be physically relevant. The perturbative argument we have in this work will implicitly ensure the non-negativity of the diffusion solution when we select appropriate point of linearization.

In the rest of the paper, we study the inverse problems in the transport regime in Section~\ref{SEC:Transport} and in the diffusion regime in Section~\ref{SEC:Diffusion}. Concluding remarks are offered in Section~\ref{SEC:Concl}. Throughout the paper, we assume that all the coefficient functions are bounded in $L^\infty(\Omega)$:
\begin{align}\label{EQ:Transport Coeff}
(a)\ \ 0 < c_0\leq \Xi (\bx), \sigma_a(\bx), \sigma_s(\bx), \sigma_b(\bx) \leq C_0, \qquad \forall \bx\in\Omega
\end{align}
for some positive constants $c_0$ and $C_0$. It is convenient in later discussion to extend these functions $\Xi (\bx), \sigma_a(\bx), \sigma_s(\bx), \sigma_b(\bx)$ by $0$ outside $\Omega$. For technical reasons, we assume further that
\begin{align}\label{EQ:Transport Coeff-1}
(b)\ \ \sigma_a\ \mbox{and}\ \sigma_s \ \mbox{are known in a $\delta$-vicinity of $\partial\Omega$ for some $\delta>0$},
\end{align}
which, in the diffusion approximation, translates to the assumption
\begin{align}\label{EQ:Transport Coeff-2}
(b')\ \ {\sigma_a}_{|\partial\Omega}\ \mbox{and}\  \gamma_{|\partial\Omega}\ \mbox{are known}.
\end{align}
While assumption $(b)$ (and therefore $(b')$) does not look harmful from the practical point of view, it is needed to ensure the correctness of the results we will present (see for instance~\cite{ReGaZh-SIAM13} for discussions on how to remove assumption $(b')$ in the diffusive regime by introducing additional data).

\section{Inverse problems in the radiative transport regime}
\label{SEC:Transport}

We start with inverse problems to the semilinear transport model~\eqref{EQ:ERT TP} with internal data of the form~\eqref{EQ:Data TP}. 
We denote by $L_{d\xi}^\infty(\Gamma_{-})$ the usual space of $L^\infty$ functions on $\Gamma_-$ with measure $d\xi=|\bnu(\bx) \cdot \bv|d\mu(\bx) d\bv$, $d\mu(\bx)$ being the surface Lebesgue measure on $\partial\Omega$.

Let us assume that we have the data encoded in the map:
\begin{equation}\label{EQ:Data Map}
\Lambda_T: g\in L_{d\xi}^{\infty}(\Gamma_{-})  \mapsto H_T\in L^\infty(\Omega).
\end{equation}
For any sufficiently small $g(\bx, \bv)\in L_{d\xi}^{\infty}(\Gamma_{-})$, the well-poseness result in Theorem~\ref{THM:WELL} ensures that there exists a unique solution $u$ to \eqref{EQ:ERT TP}. Therefore the map $\Lambda_T$ in~\eqref{EQ:Data Map} is well-defined for small $g$ in $L^{\infty}_{d\xi}(\Gamma_{-})$.

The inverse coefficient problem we are interested in solving is the following:\\[1.5ex]
{\bf Inverse Problem:} \emph{Determine the triplet $(\sigma_a, \sigma_b, \sigma_s)$ in~\eqref{EQ:ERT TP} from the data encoded in~$\Lambda_T$} defined in~\eqref{EQ:Data Map}.\\[1.5ex]
Note that theory developed in~\cite{ReZh-SIAM18} based on the diffusion approximation implies that one can not reconstruct all four coefficients $(\Xi, \sigma_a, \sigma_b, \sigma_s)$ simultaneously, no matter how much data we have. Therefore, we assume that $\Xi$ is known in the rest of the paper.

Our main strategy is to use the linearization technique of Isakov and others~\cite{Isakov-ARMA93,Isakov-CPDE01,Isakov-IP01,IsNa-TAMS95,IsSy-CPAM94} in dealing with nonlinear equations to decompose the inverse problem to the semilinear radiative transport equation~\eqref{EQ:ERT TP} into an inverse coefficient problem for the linear transport equation where we reconstruct $\sigma_a$ and $\sigma_s$ by the result of Bal-Jollivet-Jungon~\cite{BaJoJu-IP10}, and an inverse source problem for the linear transport equation where we reconstruct the two-photon absorption coefficient $\sigma_b$. This is the same type of strategy that have been successfully employed to solve many inverse problems for nonlinear PDEs recently; see for instance ~\cite{AsZh-JST20,Brander-PAMS16,CaKa-arXiv20,CaNaVa-AML19,ChLaOkPa, EgPiSc-IP14,FeOk-arXiv19,HaHyMu-IP18,HeSu-CPDE02,KaNa-IP02,KrUh-arXiv19,KrUh-PAMS20,KuLaUh, LaLi-arXiv20,LaiOhm21,LaUhYa-arXiv20,LaZh-arXiv20B,LaLiLiSa-arXiv19A,LaLiLiSa-arXiv19B,LaLiMaTy,LaUhWa18,MuUh-arXiv18,Shankar-arXiv19,Sun-MZ96,Sun-AAM04,SuUh-AJM97} and reference therein.  
\subsection{\texorpdfstring{$1^{st}$}{}-order linearization to recover \texorpdfstring{$\sigma_a$}{} and \texorpdfstring{$\sigma_s$}{}}

Let $\eps>0$ be a small parameter. We consider the following boundary value problem:
\begin{equation}\label{EQ:Semil ERT}
\begin{array}{rcll}
\bv \cdot \nabla u(\bx, \bv; \eps) + \sigma_a(\bx) u(\bx, \bv; \eps) + \sigma_b \aver{u} u(\bx, \bv; \eps) &=& \sigma_s(\bx) K (u)(\bx, \bv; \eps), &\text{in}\ X \\
u(\bx, \bv; \eps) &= & \eps g(\bx, \bv), &\text{on}\ \Gamma_{-}.
\end{array}
\end{equation}
For $g\in L_{d\xi}^\infty(\Gamma_-)$ and $\varepsilon$ sufficiently small, the boundary value problem~\eqref{EQ:Semil ERT} is well-posed according to Theorem~\ref{THM:WELL}. Moreover, the solution $u(\bx,\bv;\varepsilon)$ of \eqref{EQ:Semil ERT} satisfies $u(\bx,\bv;0)=0$ when $\varepsilon=0$ due to the well-posedness. We denote the associated data by $H_T(x;\varepsilon)$. 

Following Proposition~\ref{Pro:differentiability}, we know that $u$ is twice differentiable with respect to $\varepsilon$. Therefore, we can perform the following linearization.

Based on Proposition~\ref{Pro:differentiability}, let $u^{(1)}(\bx,\bv):=\partial_{\varepsilon}u(\bx,\bv;\varepsilon)|_{\varepsilon=0}$. By the first-order linearization, we have that $u^{(1)}$ satisfies the linear transport equation:
\begin{equation}\label{EQ:u_1 0}
\begin{array}{rcll}
\bv \cdot \nabla u^{(1)}(\bx, \bv) + \sigma_a(\bx) u^{(1)}(\bx, \bv) &=& \sigma_s(\bx) K (u^{(1)})(\bx, \bv), &\text{in}\ X \\
u^{(1)}(\bx, \bv) &= & g(\bx, \bv), &\text{on}\ \Gamma_{-}
\end{array}
\end{equation}
where we used the fact that $u(\bx,\bv;0)=0$.

For the internal data defined in \eqref{EQ:Data TP}, we also linearize it and then obtain that
\begin{equation}\label{linear internal 0}
H^{(1)}_T(\bx):=\partial_{\varepsilon}H_T(\bx;\varepsilon)|_{\varepsilon=0}=\Xi \sigma_a \aver{u^{(1)}}(\bx).
\end{equation}

It turns out that data encoded in the operator,
\begin{equation}\label{EQ:Albedo 1st}
	\Lambda_T^{(1)}: g(\bx,\bv)\in L_{d\xi}^\infty(\Gamma_-) \mapsto H^{(1)}_T\in L^\infty(\Omega),
\end{equation}
which is well-defined~\cite[Theorem 1.3]{BaChSc-SIAM16}, are sufficient to determine $\sigma_a$ and $\sigma_s$, under the assumption that $\Xi$ is known, according to a result of Bal-Jollivet-Jugnon~\cite{BaJoJu-IP10}.
\begin{proposition}[Theorem 2.6 of ~\cite{BaJoJu-IP10}]\label{prop:transport sigma a}
	Under the assumptions in~\eqref{EQ:Transport Coeff} and~\eqref{EQ:Transport Coeff-1}, the albedo operator $\Lambda_T^{(1)}$ uniquely determines $\sigma_a$ and $\sigma_s$ in $\Omega$, and the following stability holds:
	\[
		\|\sigma-\tilde\sigma\|_{W^{-1,1}(\Omega)}+\|\sigma_s-\wt\sigma_s\|_{L^1(\Omega)}\le \|\Lambda^{(1)}-\wt \Lambda^{(1)}\|_{\cL(L_{d\xi}^\infty(\Gamma_-); L^\infty(\Omega))}
	\]
	where $(\sigma:=\sigma_a+\sigma_s, \sigma_s)$ and  $(\wt\sigma:=\wt\sigma_a+\wt\sigma_s, \wt\sigma_s)$ are coefficients corresponding to $\Lambda^{(1)}$ and $\wt\Lambda^{(1)}$ respectively.
\end{proposition}
We refer interested reader to~\cite{BaJoJu-IP10} for the a more general version of this result as well as several other related stability results.
\begin{remark}\rm
Theorem 2.6 of ~\cite{BaJoJu-IP10} was derived under the framework where $\Lambda_T^{(1)}$ is viewed as an operator $\Lambda_T^{(1)}: L_{d\xi}^1(\Gamma_-) \mapsto L^1(\Omega)$. With the assumptions we have, the result in~\cite[Theorem 2.1]{BaChSc-SIAM16} ensures the well-posedness of the linear transport equation~\eqref{EQ:u_1 0} in the $L_{d\xi}^\infty(\Gamma_-) \mapsto L^\infty(X)$ framework. This allows Theorem 2.6 of ~\cite{BaJoJu-IP10} to be reproduced in the $\Lambda_T^{(1)}: L_{d\xi}^\infty(\Gamma_-) \mapsto L^1(\Omega)$ framework which then leads to Proposition~\ref{prop:transport sigma a} by standard bounds.
\end{remark}

\subsection{$2^{nd}$-order linearization to recover $\sigma_b$}

We now differentiate~\eqref{EQ:Semil ERT} twice with respect to $\varepsilon$, and obtain that   
\begin{equation}\label{EQ:ERT TP eps2}
\begin{array}{rcll}
\bv \cdot \nabla u^{(2)}(\bx, \bv) + \sigma_a(\bx) u^{(2)}(\bx, \bv) + 2 \sigma_b \aver{u^{(1)}} u^{(1)}(\bx,\bv) &=& \sigma_s(\bx) K u^{(2)}(\bx, \bv), &\text{in}\ X\\
u^{(2)}(\bx, \bv) &= & 0, &\text{on}\ \Gamma_{-}
\end{array}
\end{equation}
where the solution $u^{(2)}(\bx,\bv):=\partial_{\varepsilon}^2 u(\bx,\bv;\varepsilon)|_{\varepsilon=0}$ and $\sigma_b$ is the only to-be-recovered coefficient.
Similarly, the internal data is linearized to the second order, that is,
\begin{equation}\label{EQ:Data eps2}
H^{(2)}_T(\bx):=\partial_{\varepsilon}^2 H_T(\bx,\bv)|_{\varepsilon=0}=\Xi\Big(\sigma_a \aver{u^{(2)}}+ 2\sigma_b \aver{u^{(1)}}\aver{u^{(1)}} \Big)(\bx).
\end{equation}

From Proposition~\ref{prop:transport sigma a}, we have determined $\sigma$ and $\sigma_s$ from the first-order term in linearization. 
It remains to recover $\sigma_b$. Let $u$ and $\wt{u}$ be solutions to~\eqref{EQ:Semil ERT} with coefficients $(\sigma_a,\sigma_b,\sigma_s)$ and $(\sigma_a,\wt\sigma_b, \sigma_s)$ respectively. We denote the corresponding data by $H_T$ and $\wt H_T$. Then we have that
$
u^{(1)} = \wt{u}^{(1)}
$
and $ u^{(2)}$ and $\wt{u}^{(2)}$ are solutions to \eqref{EQ:ERT TP eps2} with $\sigma_b$ and $\wt\sigma_b$, respectively.

For any coefficient and data pair $(\sigma_a, \sigma_b, H_T)$, we define
\[
\cA_1:=\left\{(\sigma_a, \sigma_b, H_T)\mid \ \inf_\Omega \Big(\sigma_a+\bv\cdot\nabla \ln\dfrac{H_T^{(1)}}{{\Xi}\sigma_a}\Big)\ge \alpha>0\right\}
\]
for some positive constant $\alpha$, and also define
\[
\cA_2:=\{(\sigma_a, \sigma_b, H_T)\mid 0\leq \Pi  < 1 \},
\]	
where we denote
\[
\Pi :=  C_2 C_0 \|\dfrac{\Xi\sigma_a g}{H_T^{(1)}} \|_{L_{d\xi}^\infty(\Gamma_-)}\,,
\]
with the constant $C_2$ defined in Proposition~\ref{PRO:source L2} and the constant $C_0$ defined in \eqref{EQ:Transport Coeff}. 

Note that in Proposition~\ref{PRO:positive}, for suitable chosen $g\in L_{d\xi}^\infty(\Gamma_-)$, there exists a unique positive solution $u^{(1)}$ to \eqref{EQ:u_1 0} such that $u^{(1)}\geq \varepsilon'>0$ for some constant $\varepsilon'>0$ depending on $g,\Omega,\sigma_a,\sigma_s$.
We now let $\varphi=\dfrac{u^{(1)}}{\aver{u^{(1)}}}$. Then $\varphi$ solves the following transport equation:
\begin{equation*}
\begin{array}{rcll}
\bv \cdot \nabla \varphi + (\sigma_a+\bv\cdot\nabla \ln \aver{u^{(1)}})\varphi &=&\sigma_s K \varphi(\bx, \bv),  &\text{in}\ X\\
\varphi(\bx, \bv) &= &\dfrac{\Xi \sigma_a g}{H_T^{(1)}}, &\text{on}\ \Gamma_{-}\,.
\end{array}
\end{equation*}
 
\begin{lemma}\label{lemma:max bound}
	If $(\sigma_a, \sigma_b, H_T)\in \mathcal{A}_1$, then
	\begin{equation*}
	\|\varphi\|_{L^\infty(X)}\le\|\dfrac{\Xi \sigma_a g}{H_T^{(1)}}\|_{L_{d\xi}^\infty(\Gamma_-)}\,.
	\end{equation*}
\end{lemma}  

\begin{proof}
	Since $(\sigma_a, \sigma_b, H_T)\in \mathcal{A}_1$, the proof follows immediately from the maximum principle; see for instance Proposition~\ref{PRO:source L2}.
\end{proof}

We are ready to determine $\sigma_b$ provided that $(\Xi,\sigma_a,\sigma_s)$ is known. More precisely, we have the following result.
	\begin{theorem}
		Let $H_T$ and $\wt H_T$ be the internal data corresponding to the coefficient sets $(\Xi,\sigma_a,\sigma_b,\sigma_s)$ and $(\Xi, \sigma_a,\wt\sigma_b, \sigma_s)$, both satisfying~\eqref{EQ:Transport Coeff}, respectively. Assume that the coefficient-datum pairs $(\sigma_a, \sigma_b, H_T)$ and $(\sigma_a, \wt \sigma_b, \wt H_T)$ are both in the class of $\cA_1\cap\cA_2$. Then $\sigma_b$ and $\wt \sigma_b$ can be reconstructed from $H^{(2)}_T$ and $\wt H^{(2)}_T$, that is, 
		\begin{equation}\label{EQ:Stab}
		\|(\sigma_b-\wt\sigma_b)\aver{u^{(1)}} u^{(1)}\|_{L^2(X)}\le C \|H^{(2)}_T-\wt H^{(2)}_T\|_{L^2(\Omega)}
		\end{equation}
		for some constant $C=\frac{ 1}{2 c_0 (1- \Pi)} \|\dfrac{\Xi\sigma_a g}{H_T^{(1)}} \|_{L_{d\xi}^\infty(\Gamma_-)} \geq 0$.
		
		Moreover, due to the positive lower bound of $u^{(1)}$, we have
		\begin{equation}\label{EQ:Stab sigma_b}
		\|\sigma_b-\wt\sigma_b\|_{L^2(\Omega)}\le C \|H^{(2)}_T-\wt H^{(2)}_T\|_{L^2(\Omega)}.
		\end{equation}
	\end{theorem}	

\begin{proof}
	From the data~\eqref{EQ:Data eps2} and the fact that $u^{(2)}$, and $\wt{u}^{(2)}$ are solutions to \eqref{EQ:ERT TP eps2} with the same $\sigma_a$, we have that
	\begin{align}\label{EQ:Sigmab Stab 1}
	&	\|2(\sigma_b-\wt\sigma_b) \aver{u^{(1)}} u^{(1)} \|_{L^2(X)} \notag \\
	& \le \|{u^{(1)} \over \aver{u^{(1)}}}\|_{L^\infty(X)} \|\frac{H_T^{(2)}-\wt H_T^{(2)}}{\Xi}\|_{L^2(\Omega)}+\| \sigma_a {u^{(1)} \over \aver{u^{(1)}}}  (\aver{u^{(2)}}-\aver{\wt u^{(2)}})\|_{L^2(X)}\notag\\
	&\le  {1\over c_0}\|{u^{(1)} \over \aver{u^{(1)}}}\|_{L^\infty(X)}   \|H_T^{(2)}-\wt H_T^{(2)}\|_{L^2(\Omega)}+\| \sigma_a{u^{(1)} \over \aver{u^{(1)}}}\|_{L^\infty(X)}\|\aver{u^{(2)}-\wt u^{(2)}}\|_{L^2(\Omega)}.
	\end{align}
	We observe also that, for any $\phi(\bx, \bv) \in L^2(X)$, by Jensen's inequality, we have that
	\begin{align}\label{EQ:General Bound}
	\|\aver{\phi}\|_{L^2(\Omega)}^2 
	\le  \| \phi\|_{L^2(X)}^2.
	\end{align}	
	Therefore, ~\eqref{EQ:Sigmab Stab 1} can be written as
	\begin{align}\label{EQ:Sigmab Stab 2}
	&\|2(\sigma_b-\wt\sigma_b) \aver{u^{(1)}} u^{(1)}\|_{L^2(X)} \notag\\ 
	&\le  {1\over c_0} \|{u^{(1)} \over \aver{u^{(1)}}}\|_{L^\infty(X)}  \|H_T^{(2)}-\wt H_T^{(2)}\|_{L^2(\Omega)}+\| \sigma_a {u^{(1)} \over \aver{u^{(1)}}} \|_{L^\infty(X)} \|u^{(2)}-\wt u^{(2)}\|_{L^2(X)}.
	\end{align}
	
	Let $w=u^{(2)}-\wt u^{(2)}$. We verify that $w$  solves the transport equation
	\begin{equation*}
	\begin{array}{rcll}
	\bv \cdot \nabla w(\bx, \bv) + \sigma_a(\bx) w &=& \sigma_s(\bx) K w-2(\sigma_b-\wt\sigma_b)\aver{u^{(1)}} u^{(1)},  &\text{in}\ X\\
	w(\bx, \bv) &= & 0,  &\text{on}\ \Gamma_{-}\,.
	\end{array}
	\end{equation*}
	Therefore, we have that, for some constant $C_2>0$ in Proposition~\ref{PRO:source L2},
	\begin{equation}\label{EQ:Sigmab Stab 3}
	\|w\|_{L^2(X)}\le C_2\|2(\sigma_b-\wt\sigma_b)\aver{u^{(1)}} u^{(1)}\|_{L^2(X)}.
	\end{equation}
	By Lemma~\ref{lemma:max bound}, we have
	$$
	C_2\| \sigma_a {u^{(1)} \over \aver{u^{(1)}}} \|_{L^\infty(X)} \leq C_2 C_0  \|\dfrac{\Xi\sigma_a g}{H_T^{(1)}}  \|_{L_{d\xi}^\infty(\Gamma_-)}=\Pi < 1
	$$
	provided that $(\sigma_a, \sigma_b, H_T)\in \mathcal{A}_2$.
	Hence,~\eqref{EQ:Sigmab Stab 2} and~\eqref{EQ:Sigmab Stab 3} lead to 
	\begin{equation*}
	\|(\sigma_b-\wt\sigma_b) \aver{u^{(1)}} u^{(1)}\|_{L^2(X)} \le \frac{1}{2  c_0  (1- \Pi)} \|\dfrac{\Xi\sigma_a g}{H_T^{(1)}} \|_{L_{d\xi}^\infty(\Gamma_-)} \|H_T^{(2)}-\wt H_T^{(2)}\|_{L^2(\Omega)}.
	\end{equation*}	 
	This completes the proof.
\end{proof}

\begin{remark}
	To reconstruct $\sigma_b$, we have to make the assumption that the coefficient-datum triplet $(\sigma_a, \sigma_b, H_T)$ satisfies the constraints in $\cA_1$ and $\cA_2$. We do not have a precise characterization of the coefficients and the boundary conditions needed to make the constraints realizable at the moment. However, in the linearization technique, we reconstruct $(\sigma_a, \sigma_s)$ before we reconstruct $\sigma_b$. With $(\sigma_a, \sigma_s)$ known, it seems possible to select boundary conditions, following the constructions of~\cite{BaChSc-SIAM16}, to have the transport solution $u^{(1)}$ with small gradient relative to its size  so that $\cA_1$,  which is equivalent to $\inf_\Omega \big(\sigma_a+\bv\cdot\nabla \ln\aver{u^{(1)}} \big)\ge \alpha>0$,  is achievable. In the regime of practical applications, we have $\sigma_a\ll \sigma_s$. In such a case, $\cA_2$ roughly simplifies to $\frac{C_0}{c_0}\|\frac{g}{\aver{g}+\aver{u^{(1)}_{\Gamma_+}}}\|_{L_{d\xi}^\infty(\Gamma_-)}<1$. This might be achievable when the contrast of $\sigma_a$ is small (that is, $C_0/c_0$ is close to $1$), in which case we try to select isotropic boundary sources that generate solutions with large outgoing component, $u^{(1)}_{|\Gamma_+}$, on the boundary. It is of great interest, both on the technical and on the practical aspects, to see if one can find methods to relax (or even remove) the assumptions $\cA_1$ and $\cA_2$.
\end{remark}

\subsection{A result on uncertainty quantification in transport regime}

Our result in the previous section allows us to reconstruct all three coefficient $\sigma_a$, $\sigma_b$ and $\sigma_s$ when we have data encoded in the full operator $\Lambda_T$. In practical applications, one might only have a limited number of data sets to use. In such cases, it is not realistic trying to reconstruct all the coefficients. In many biological imaging applications, the absorption coefficients are of great interests since they are very sensitive to pathological changes in tissues while the scattering coefficient $\sigma_s$ is much less sensitive. One therefore often tries to reconstruct $\sigma_a$ and $\sigma_b$ assuming $\sigma_s$ is known. An important issue in this approach is to characterize the impact of the inaccuracy in the value of $\sigma_s$ on the reconstruction of $(\sigma_a, \sigma_b)$. In the next theorem, we give a sensitivity result for such an uncertainty quantification issue.

\begin{theorem}\label{THM:Uncertain ERT}
		Let $(\sigma_a, \sigma_b)$ and $(\wt\sigma_a, \wt\sigma_b)$ be reconstructed with $\sigma_s$ and $\wt\sigma_s$ respectively, from the same data set $H_T$. Assume that the coefficient data pairs $(\sigma_a, \sigma_b, H_T)$ and $(\wt \sigma_a, \wt \sigma_b, H_T)$ are both in the class of $\cA_1\cap\cA_2$. Then we have that,
		\begin{equation}\label{EQ:Uncertainty ERT}
		\|\sigma_a-\wt\sigma_a\|_{L^2(\Omega)}+\|\sigma_b-\wt\sigma_b\|_{L^2(\Omega)} \le \fc \|\sigma_s-\wt\sigma_s\|_{L^2(\Omega)}
		\end{equation}
		for some constant $\fc>0$.
\end{theorem}
\begin{proof}
	\noindent \textbf{(1). Estimate for $\sigma_a$.} We start with the problem of reconstructing $\sigma_a$ from the first-order data $H^{(1)}_T$. Since the same data set is used for the reconstructions, we have that
	\begin{equation*}
	\Xi \sigma_a \aver{u^{(1)}} = \Xi \wt\sigma_a \aver{\wt u^{(1)}}=H^{(1)}_T.
	\end{equation*}
	This leads to the equality
	\begin{equation}\label{EQ:UQ Data 1}
	(\sigma_a-\wt\sigma_a)u^{(1)}=\wt\sigma_a\frac{u^{(1)}}{\aver{u^{(1)}}} \aver{\wt u^{(1)}-u^{(1)}},
	\end{equation}
	which thus gives the bound
	\begin{align}\label{EQ:UQ Data Bound 1}
	\|(\sigma_a-\wt\sigma_a)u^{(1)}\|_{L^2(X)}
	&\le \|\wt\sigma_a\frac{u^{(1)}}{\aver{u^{(1)}}}\|_{L^\infty(\Omega)} \|\aver{\wt u^{(1)}-u^{(1)}}\|_{L^2(\Omega)}\notag\\ 
	&\le  C_0 \| \frac{u^{(1)}}{\aver{u^{(1)}}}\|_{L^\infty(\Omega)}\|\wt u^{(1)}-u^{(1)}\|_{L^2(X)},
	\end{align}
	where the last follows from~\eqref{EQ:General Bound}.
	
	Let us define $\wt w:=u^{(1)}- \wt u^{(1)}$. Then $\wt w$ solves the following transport equation:
	\begin{equation*}
	\begin{array}{rcll}
	\bv \cdot \nabla \wt w+ \wt\sigma_a(\bx) \wt w &=&\wt \sigma_s(\bx) K \wt w(\bx, \bv)-(\sigma_a-\wt\sigma_a) u^{(1)}+(\sigma_s-\wt\sigma_s)K(u^{(1)}), &\text{in}\ X\\
	\wt w(\bx, \bv) &= &0,  &\text{on}\ \Gamma_{-}\,.
	\end{array}
	\end{equation*}
	This equation gives us that, for some constant $C_2>0$ as in Proposition~\ref{PRO:source L2},
	\begin{equation}\label{EQ:UQ Stab 1}
	\|\wt w\|_{L^2(X)}\le C_2 \Big( \|(\sigma_a-\wt\sigma_a)u^{(1)}\|_{L^2(X)} +\|(\sigma_s-\wt\sigma_s)K(u^{(1)})\|_{L^2(X)}\Big).
	\end{equation}
	The combination of~\eqref{EQ:UQ Data Bound 1} and~\eqref{EQ:UQ Stab 1} then implies the bound:
	\begin{align}\label{EQ:UQ Stab 2}
	&\|(\sigma_a-\wt\sigma_a)u^{(1)}\|_{L^2(X)} \notag\\ 
	&\le  C_2C_0 \|\frac{u^{(1)}}{\aver{u^{(1)}}}\|_{L^\infty(\Omega)}\Big( \|(\sigma_a-\wt\sigma_a)u^{(1)}\|_{L^2(X)} +\|(\sigma_s-\wt\sigma_s)K(u^{(1)})\|_{L^2(X)}\Big).
	\end{align}	
	This, together with the assumption that 
	$$
	\Pi := C_2 C_0 \|\dfrac{\Xi \sigma_a g}{H^{(1)}_T}\|_{L^\infty(\Gamma_-)} < 1,
	$$
	leads to the bound
	\begin{equation}\label{EQ:UQ Stab 3}
	\|(\sigma_a-\wt\sigma_a)u^{(1)}\|_{L^2(X)} \\ 
	\le {\Pi \over 1 - \Pi} \|(\sigma_s-\wt\sigma_s)K(u^{(1)})\|_{L^2(X)}.
	\end{equation}
	Since $u^{(1)}$ is positive and bounded away from zero, we thus have
	\begin{equation}\label{EQ:UQ Stab 4 2}
	\|\sigma_a-\wt\sigma_a\|_{L^2(\Omega)} \le \fc_1 \|\sigma_s-\wt\sigma_s \|_{L^2(\Omega)}.
	\end{equation}
	\noindent \textbf{(2). Estimate for $\sigma_b$.} In a similar manner, we can bound the uncertainty in the reconstruction of $\sigma_b$ with the uncertainty in $\sigma_s$. We again start with the fact that the same data set is used in the reconstructions with different $\sigma_s$. This leads to the relation:
	\begin{equation*}
	\sigma_a \aver{u^{(2)}} + 2\sigma_b \aver{u^{(1)}}^2 = \wt\sigma_a \aver{\wt u^{(2)}} + 2\wt\sigma_b \aver{\wt u^{(1)} }^2 = H^{(2)}_T/\Xi.
	\end{equation*}
	This relation gives us the bound:
	\begin{align}\label{EQ:UQ Stab 5}
	&\|2 (\sigma_b-\wt\sigma_b) \aver{u^{(1)}}u^{(1)}\|_{L^2(X)} \notag\\
	&\le \|(\wt\sigma_a-\sigma_a){u^{(1)} \over \aver{u^{(1)}}} \aver{u^{(2)}}\|_{L^2(X)} +
	\| 2\wt\sigma_b {u^{(1)} \over \aver{u^{(1)}}} ( \aver{u^{(1)}}^2 - \aver{\wt u^{(1)}}^2)\|_{L^2(X)} \notag\\
	&\quad + \|\wt\sigma_a {u^{(1)} \over \aver{u^{(1)}}} ( \aver{u^{(2)}} - \aver{\wt u^{(2)}})\|_{L^2(X)} =: I_1+I_2+I_3.
	\end{align}
	
	To estimate $I_1$ and $I_2$, we apply \eqref{EQ:General Bound}, \eqref{EQ:UQ Stab 1}, and \eqref{EQ:UQ Stab 4 2} to get that
	\begin{align*}
	I_1+I_2&=\|(\wt\sigma_a-\sigma_a){u^{(1)} \over \aver{u^{(1)}}} \aver{u^{(2)}}\|_{L^2(X)} +
	\| 2\wt\sigma_b {u^{(1)} \over \aver{u^{(1)}}} ( \aver{u^{(1)}}^2 - \aver{\wt u^{(1)}}^2)\|_{L^2(X)}\\
	&\leq \fc_1 \|\sigma_s-\wt\sigma_s \|_{L^2(\Omega)}.
	\end{align*}
	
	To estimate $I_3$, we only need to control the term $\| u^{(2)} - \wt u^{(2)}\|_{L^2(X)}$. 
	Let $\wh w:=u^{(2)}-\wt u^{(2)}$. Then $w$ solves:
	\begin{equation*}
	\begin{array}{rcll}
	\bv \cdot \nabla \wh w +\sigma_a \wh w  &=& \sigma_s K \wh w - (\sigma_a -\wt\sigma_a)\wt u^{(2)} + (\sigma_s-\wt\sigma_s)K(\wt u^{(2)}) + 2\wt\sigma_b\aver{\wt u^{(1)}}\wt u^{(1)} - 2 \sigma_b\aver{ u^{(1)}}u^{(1)}, &\text{in}\ X\\
	\wh w(\bx, \bv) &= &0,  &\text{on}\ \Gamma_{-}\,.
	\end{array}
	\end{equation*}
	From Proposition~\ref{PRO:source L2}, we have
	\begin{align}\label{EQ:UQ Stab 6}
	\|\wh w\|_{L^2(X)}&\le C_2\Big( \|(\sigma_a -\wt\sigma_a)\wt u^{(2)}\|_{L^2(X)}  + \|(\sigma_s-\wt\sigma_s)K(\wt u^{(2)})\|_{L^2(X)} + \|2\wt\sigma_b (\aver{\wt u^{(1)}}\wt u^{(1)} -  \aver{ u^{(1)}}u^{(1)})\|_{L^2(X)}  \notag \\
	&\quad + 2\| (\wt\sigma_b  - \sigma_b)\aver{u^{(1)}}u^{(1)}\|_{L^2(X)}\Big).
	\end{align}
	In particular, the first three terms on the right-hand side of \eqref{EQ:UQ Stab 6} are bounded by $\|\sigma_s-\wt\sigma_s\|$ only. 
	This yields that 
	\begin{align*}
	I_3 &= \|\wt\sigma_a {u^{(1)} \over \aver{u^{(1)}}} ( \aver{u^{(2)}} - \aver{\wt u^{(2)}})\|_{L^2(X)} \\
	&\leq \|\wt\sigma_a {u^{(1)} \over \aver{u^{(1)}}}\|_{L^\infty(X)} \|u^{(2)} - \wt u^{(2)}\|_{L^2(X)} \\	
	&\leq \fc_1 \|\sigma_s-\wt\sigma_s \|_{L^2(\Omega)} +2 C_2 C_0 \|{u^{(1)} \over \aver{u^{(1)}}}\|_{L^\infty(X)} \| (\wt\sigma_b  - \sigma_b)\aver{u^{(1)}}u^{(1)}\|_{L^2(X)}.
	\end{align*}
	From \eqref{EQ:UQ Stab 5} and estimates for $I_1,I_2,I_3$, we finally have
	\begin{align}\label{EQ:UQ Stab 5 2}
	&\|2 (\sigma_b-\wt\sigma_b) \aver{u^{(1)}}u^{(1)}\|_{L^2(X)} \notag\\
	&\le  \fc_1 \|\sigma_s-\wt\sigma_s \|_{L^2(\Omega)} +2 C_2 C_0 \|{u^{(1)} \over \aver{u^{(1)}}}\|_{L^\infty(X)} \| (\wt\sigma_b  - \sigma_b)\aver{u^{(1)}}u^{(1)}\|_{L^2(X)}.
	\end{align}
	We can now apply again the hypothesis
	$$
    C_2 C_0 \|{u^{(1)} \over \aver{u^{(1)}}}\|_{L^\infty(X)} \leq \Pi < 1,
	$$
	to obtain that
	$$
	\|(\sigma_b-\wt\sigma_b) \aver{u^{(1)}}u^{(1)}\|_{L^2(X)}\leq \frac{\fc_1 }{2(1-\Pi)}\|\sigma_s-\wt\sigma_s\|_{L^2(\Omega)}.
	$$
	The factor $\aver{u^{(1)}}u^{(1)}$ can again be removed using the fact that $u^{(1)}$ is positive and bounded away from zero. The proof is complete.
\end{proof}  
The above result says that the reconstruction of $(\sigma_a, \sigma_b)$ is reliable if we do not make a large error in the scattering coefficient $\sigma_s$ we assumed in the reconstruction.

\section{Inverse problems in the diffusive regime}
\label{SEC:Diffusion}

We reproduce the results in the previous section in the diffusive regime. Throughout this section, we make the following assumptions on the coefficients:
\begin{align}\label{EQ:Diffusion Coeff}
	\begin{array}{c}
		\Xi, \gamma(\bx), \sigma_a(\bx), \sigma_b(\bx) \in \cC^{2}(\overline\Omega)\\[1ex]
		0 < c_0\leq {\|\Xi\|_{\cC^{2}(\overline\Omega)}, \|\gamma\|_{\cC^{2}(\overline\Omega)}, \|\sigma_a\|_{\cC^{2}(\overline\Omega)}, \|\sigma_b\|_{\cC^{2}(\overline\Omega)}} \leq C_0,  
	\end{array}
\end{align} 
for some constants $c_0,\,C_0>0$. Under this assumption, it is shown in Theorem~\ref{THM:Wellposed Diff} that there exists a unique solution $u\in W^{2,p}(\Omega)$ to~\eqref{EQ:Diff TP} with Dirichlet boundary condition $g\in W^{2-1/p,p}(\p\Omega)$ for small enough $g$.  
In fact, it is straightforward to verify that
\[
	\begin{split}
		\|H_D\|_{L^p(\Omega)}\leq C\left(\|u\|_{L^p(\Omega)}+\|u\|_{L^\infty(\Omega)}\|u\|_{L^p(\Omega)}\right)&\leq C\left(1+\|u\|_{W^{2,p}(\Omega)}\right)\|u\|_{W^{2,p}(\Omega)}. 
		\end{split}
\]
Note that since $u\in W^{2,p}(\Omega)$, Sobolev embedding yields $u\in \cC^{1,1-d/p}(\overline\Omega)$. Then we have 
\[\begin{split}
\|\nabla H_D\|_{L^p(\Omega)}\leq & C\left\|\nabla (\sigma_a u)+u^2 \nabla\sigma_b+\sigma_b\nabla(u^2)\right\|_{L^p(\Omega)} + C\|\sigma_a u+\sigma_b u^2\|_{L^p(\Omega)} \\
\leq& C\left(\|u\|_{L^p(\Omega)}+\|\nabla u\|_{L^p(\Omega)}+\|u\|_{L^\infty(\Omega)}\|u\|_{L^p(\Omega)}+\|\nabla u\|_{L^p(\Omega)}\|u\|_{L^\infty(\Omega)}\right)\\
\leq& C\left(1+\|u\|_{W^{2,p}(\Omega)}\right)\|u\|_{W^{2,p}(\Omega)}.
\end{split}\]
Similarly, we can also show the second derivatives satisfy
\[\|\partial_{jk}H_D\|_{L^p(\Omega)}\leq C\left(1+\|u\|_{W^{2,p}(\Omega)}\right)\|u\|_{W^{2,p}(\Omega)}\quad \textrm{ for }j, k=1,\ldots,d.\]
Therefore, we have
\[
	\|H_D\|_{W^{2,p}(\Omega)}\leq C\left(1+\|f\|_{W^{2-1/p,p}(\Omega)}\right)\|f\|_{W^{2-1/p,p}(\Omega)}.
\]
This shows that for $g$ sufficiently small, the data encoded in the map 
\begin{equation}\label{EQ:Data Map Diff}
	\Lambda^D: g\in W^{2-1/p,p}(\partial\Omega) \mapsto H_D \in {W^{2,p}(\Omega)},
\end{equation}
are well-defined.

The inverse coefficient problem we are interested in solving is the following:\\[1.5ex]
{\bf Inverse Problem:} \emph{Reconstruct the triplet $(\gamma, \sigma_a, \sigma_b)$ in~\eqref{EQ:Diff TP} from data encoded in~$\Lambda^D$} defined in~\eqref{EQ:Data Map Diff}.\\[1.5ex]
This problem has been investigated in~\cite{ReZh-SIAM18} where uniqueness and stability are established for the problem linearized around a known background coefficient.  

\subsection{The reconstruction of $(\gamma, \sigma_a, \sigma_b)$}

We conduct higher-order linearization steps to the following boundary value problem with $g\in W^{2-1/p,p}(\partial\Omega)$ and small $\eps>0$:
\begin{align}\label{EQN:diffusion}
\begin{array}{rcll}
	-\nabla \cdot\gamma \nabla u(\bx;\varepsilon) +\sigma_a(\bx) u(\bx;\varepsilon) +\sigma_b(\bx) u(\bx;\varepsilon) u(\bx;\varepsilon) &=&0, & \mbox{in}\ \ \Omega\\
	u(\bx;\varepsilon)&=& \varepsilon g(\bx), & \mbox{on}\ \partial\Omega.
\end{array}
\end{align}
	Indeed one can show that $u(\bf x;\varepsilon)$ is twice differentiable with respect to $\varepsilon$ by following a similar argument as in the proof of Proposition~\ref{Pro:differentiability} for the transport equation. Therefore one can perform the following linearizations. 

Denote the associated internal data by $H_D(\bf x;\varepsilon)$. By the first-order linearization, we have $u^{(1)}:=\partial_{\varepsilon}u|_{\varepsilon=0}$ satisfying the linear diffusion equation:
\begin{align}\label{linear diffusion equation}
\begin{array}{rcll}
-\nabla \cdot\gamma \nabla u^{(1)}(\bx)+\sigma_a u^{(1)}(\bx) &=&0, & \hbox{in}\ \ \Omega\\
u^{(1)}(\bx) &=& g(\bx),  &\hbox{on}\ \partial\Omega\,.
\end{array}
\end{align}
For the internal data, we also linearize it and then obtain that
\begin{equation}\label{linear internal}
{H}^{(1)}_D(\bx):=\p_{\varepsilon}H_D(\bx;\varepsilon)|_{\varepsilon=0}=\Xi\sigma_a u^{(1)}(\bx).
\end{equation}

When $\Xi$ is known, we can apply the result in~\cite{BaUh-IP10,BaRe-IP11} to obtain the following lemma. 
\begin{proposition}[\cite{BaUh-IP10,BaRe-IP11}]\label{recover gamma alpha}
	Under the assumptions in~\eqref{EQ:Transport Coeff} and~\eqref{EQ:Transport Coeff-2}, there exists a pair of boundary conditions $(g_1, g_2)$ such that the coefficient pair $(\gamma, \sigma_a)$ is uniquely determined by the linearized internal data $({H}^{(1)}_{D,1}, {H}^{(1)}_{D,2})$.
\end{proposition} 
The construction of the boundary condition pair $(g_1, g_2)$ is highly non-trivial. We refer to~\cite{BaUh-IP10,BaRe-IP11} for the technical details and ~\cite{AlDiFrVe-AdM17} for an alternative approach to relax some of the strong conditions needed for the theory to work. Note also that with the assumption that $H_D$ is known on the boundary $\partial\Omega$, ${\sigma_a}_{|\partial\Omega}$ can be reconstructed by $H_D^{(1)}/(\Xi g)$. This would allow us to remove the assumption that ${\sigma_a}_{|\partial\Omega}$ is known on the boundary in the diffusive regime.

Next we perform the second linearization. Set 
$$
u^{(2)}(\bx):=\p^2_{\varepsilon}u(\bx;\varepsilon)|_{\varepsilon=0}.
$$
It satisfies
\begin{align}\label{EQN:2nd linearization}
\begin{array}{rcll}
-\nabla \cdot\gamma \nabla u^{(2)}(\bx)+\sigma_a u^{(2)}(\bx)+{2}\sigma_b  u^{(1)} u^{(1)} (\bx)&=& 0, & \hbox{in}\ \ \Omega\\
u^{(2)}(\bx) &=& 0,  &\hbox{on}\ \partial\Omega\,.
\end{array}
\end{align} 
The second order linearization of the internal data gives
$$
   H^{(2)}_D(\bx):=\p_{\varepsilon}^2 H_D(\bx;\varepsilon)|_{\varepsilon=0}=\Xi\LC\sigma_a u^{(2)}+ {2}\sigma_b u^{(1)} u^{(1)}\RC(\bx).
$$

From Proposition~\ref{recover gamma alpha}, the coefficients $\gamma$ and $\sigma_a$ have been uniquely recovered in the first linearization.
Hence it remains to recover $\sigma_b$, which appears in the source term in \eqref{EQN:2nd linearization}. To this end, let $u$ and $\wt{u}$ be solutions to~\eqref{EQN:diffusion} with coefficients $(\gamma,\sigma_a,\sigma_b)$ and $(\gamma, \sigma_a,\wt\sigma_b)$ respectively. We denote the corresponding data by $H_D$ and $\wt H_D$. Then the first differentiation of $u$ and $\wt u$ satisfy
$
u^{(1)} = \wt{u}^{(1)}
$
and also $ u^{(2)}$, and $\wt{u}^{(2)}$ are solutions to \eqref{EQN:2nd linearization} with $\sigma_b$ and $\wt\sigma_b$, respectively.

Then we have the following stability result for $\sigma_b$. 

\begin{theorem}	Let $H_D$ and $\wt H_D$ be the internal data corresponding to the coefficient sets $(\gamma,\sigma_a,\sigma_b)$ and $(\gamma, \sigma_a,\wt\sigma_b)$, both satisfying~\eqref{EQ:Diffusion Coeff}, respectively. Then we have 
\begin{equation}\label{eqn:diff_stability}
 \|(\sigma_b-\wt\sigma_b)(u^{(1)})^2\|_{L^p(\Omega)}\leq C\|H^{(2)}_D-\wt H^{(2)}_D\|_{W^{2,p}(\Omega)}.
\end{equation}
If, in addition, we have that $\underline{g}:=\displaystyle\inf_{\partial\Omega} g>0$, 
then 
\begin{equation}\label{eqn:diff_stability_mub} 
	 \|\sigma_b-\wt\sigma_b \|_{L^p(\Omega)}\leq C\|H^{(2)}_D-\wt H^{(2)}_D\|_{W^{2,p}(\Omega)}
\end{equation}
where the constant $C>0$ depends on $\Omega,\gamma,\Xi$ and $\sigma$.
\end{theorem}
\begin{proof}
	Let $U:=\frac{H^{(2)}_D}{\Xi\sigma_a}=u^{(2)}+{2\sigma_b\over\sigma_a} u^{(1)} u^{(1)}$. It is a known $W^{2,p}(\Omega)$ function in $\Omega$ since $H^{(2)}_D,\Xi,\sigma_a$ are known. Let $\psi= {2\sigma_b\over \sigma_a} u^{(1)} u^{(1)}$. It solves the following problem:	
	\begin{equation}\label{EQN:psi}
		\begin{array}{rcll}
			 -\nabla \cdot\gamma  \nabla \psi &=& -\nabla \cdot\gamma  \nabla U+\sigma_a U,  &\text{in}\ \Omega\\
			 \psi&= & U, &\text{on}\ \p\Omega\,.
		\end{array}
	\end{equation}
    Since $U$, $\gamma$ and $\sigma_a$ are all known, solving the boundary value problem \eqref{EQN:psi} recovers $\psi$ in $\Omega$. Therefore, we can recover $\sigma_b$ at the point where $u^{(1)}$ is not vanishing. More precisely, reconstructing $\sigma_b$ through $\sigma_b = \psi \sigma_a/(2u^{(1)}u^{(1)})$.
    Indeed given a nonzero boundary condition $g$, by the unique continuation,  
    the set of points in $\Omega$ where $u^{(1)}=0$ has measure zero. This shows that $H^{(2)}_D$ determines $\sigma_b$.

To prove the stability estimates, we use the fact that
\begin{equation}
	\begin{array}{rcll}
	-\nabla \cdot\gamma  \nabla \LC {2(\sigma_b-\wt\sigma_b)\frac{u^{(1)}u^{(1)}}{\sigma_a}}\RC&=&-\nabla\cdot\gamma\nabla\left(\frac{H^{(2)}_D-\wt H^{(2)}_D}{\Xi\sigma_a}\right)+\frac1{\Xi}(H^{(2)}_D-\wt H^{(2)}_D), &\text{in}\ \Omega\\
	{2(\sigma_b-\wt\sigma_b)\frac{u^{(1)}u^{(1)}}{\sigma_a}}&=&0,  &\text{on}\ \p\Omega\,,
		\end{array}
\end{equation}
and elliptic regularity to have
\begin{align}\label{EST: sigma a}
\begin{split}
	\left\|(\sigma_b-\wt\sigma_b)\frac{u^{(1)}u^{(1)}}{\sigma_a}\right\|_{W^{2,p}(\Omega)}&\leq  C\left\|-\nabla\cdot\gamma\nabla\left(\frac{H^{(2)}_D-\wt H^{(2)}_D}{\Xi\sigma_a}\right)+\frac1{\Xi}(H^{(2)}_D-\wt H^{(2)}_D)\right\|_{L^p(\Omega)}\\
	&\leq C\|H_D^{(2)}-\wt H_D^{(2)}\|_{W^{2,p}(\Omega)}.
\end{split}.
\end{align}
This proves~\eqref{eqn:diff_stability}. 
	
When we have additionally that $\underline{g}:=\displaystyle\inf_{\partial\Omega} g>0$, we conclude from~\cite[Proof of Claim 4.2]{AlDiFrVe-AdM17} (see also a summary in~\cite[Theorem 2.4]{ReZh-SIAM18}) that
	   $$u^{(1)}\geq \varepsilon'>0 $$
	for some constant $\varepsilon'>0$. Together with the boundedness of $\sigma_a$, this allows us to remove the factor $\frac{u^{(1)} u^{(1)}}{\sigma_a}$ in~\eqref{EST: sigma a} to get~\eqref{eqn:diff_stability_mub}.  
\end{proof}

\subsection{Parametric uncertainty in diffusive regime.} 

We consider here the stability of reconstructing $(\sigma_a, \sigma_b)$ with respect to changes in the diffusion coefficient $\gamma$.

We first derive the following estimates, which will be applied later to show the uncertainty result.
\begin{lemma}\label{lemma:diffusion u1}
   Let $H_D$ be the internal function associated with both $(\Xi,\gamma,\sigma_a,\sigma_b)$ and $(\Xi,\wt\gamma, \wt\sigma_a,\wt\sigma_b)$.
	Then we have
	\begin{equation}\label{eqn:est_u1}
	\|u^{(1)}-\wt u^{(1)} \|_{W^{2,p}(\Omega)} \le C \left\|\frac{\wt\gamma-\gamma}{\wt \gamma}\right\|_{W^{1,p}(\Omega)},
	\end{equation}
	and 
	\begin{equation}\label{eqn:est_u2}
	\|u^{(2)}-\wt u^{(2)} \|_{W^{2,p}(\Omega)} \le C\left\|\frac{\wt\gamma-\gamma}{\wt \gamma}\right\|_{W^{1,p}(\Omega)},
	\end{equation}
for some positive constant $C$ depending on $\Omega, \gamma, g$.
\end{lemma}
\begin{proof} 
First, we have $\sigma_a u^{(1)}=\wt\sigma_a\wt u^{(1)}$. Let $w:=u^{(1)}- \wt u^{(1)}$. Then $w$ solves the diffusion equation:
    \begin{equation*}
    -\nabla \cdot \gamma \nabla w =\nabla\cdot(\gamma-\wt\gamma)\nabla \wt u^{(1)},\quad \mbox{in}\ \Omega, \qquad w=0, \quad \mbox{on}\ \partial\Omega.
    \end{equation*}
    This leads to the fact that, for some constant $C>0$ depending on $\Omega$ and $\gamma$,
	\begin{equation}\label{EQ:Stab F}
	\|u^{(1)}-\wt u^{(1)}\|_{W^{2,p}(\Omega)}\le C\|\nabla\cdot(\gamma-\wt\gamma)\nabla \wt u^{(1)}\|_{L^p(\Omega)}.
	\end{equation}
	Following~\cite{ReVa-SIAM20}, we verify that:
	\begin{align*}
	\nabla\cdot(\wt\gamma-\gamma)\nabla \wt u^{(1)}&=\nabla\cdot\frac{\wt\gamma-\gamma}{\wt\gamma} \wt\gamma\nabla \wt u^{(1)} = \frac{\wt\gamma-\gamma}{\wt\gamma}\nabla\cdot\wt\gamma\nabla \wt u^{(1)} + \wt\gamma\nabla \wt u^{(1)}\cdot\nabla \frac{\wt\gamma-\gamma}{\wt\gamma} \\ 
	&= \dfrac{\wt H^{(1)}_D}{\Xi} \frac{\wt\gamma-\gamma}{\wt\gamma} + \wt\gamma\nabla \wt u^{(1)} \cdot\nabla \frac{\wt\gamma-\gamma}{\wt\gamma}.
	\end{align*}
	This implies that
	\begin{align*}
	\|\nabla\cdot (\wt\gamma-\gamma)\nabla \wt u^{(1)}\|_{L^p(\Omega)} 
	&\le C\left(\left\|\wt H_D^{(1)}\right\|_{L^\infty(\Omega)}\left\|\frac{\wt\gamma-\gamma}{\wt\gamma}\right\|_{L^p(\Omega)}+ \|\nabla \wt u^{(1)}\|_{L^\infty(\Omega)} \left\|\nabla \frac{\wt\gamma-\gamma}{\wt\gamma}\right\|_{L^p(\Omega)}\right)\\
	&\le C\left(\left\|\wt H_D^{(1)}\right\|_{W^{2,p}(\Omega)}+ \|\wt u^{(1)}\|_{W^{2,p}(\Omega)}\right) \left\|\frac{\wt\gamma-\gamma}{\wt\gamma}\right\|_{W^{1,p}(\Omega)}\\
	&\le C\left\|g\right\|_{W^{2-1/p,p}(\p\Omega)}  \left\|\frac{\wt\gamma-\gamma}{\wt\gamma}\right\|_{W^{1,p}(\Omega)}
	\end{align*}
for some constant $C$. This can be combined with~\eqref{EQ:Stab F} to obtain~\eqref{eqn:est_u1}. Meanwhile, we can verify that
\[
	-\nabla\cdot \gamma \nabla(u^{(2)}-\wt u^{(2)})=\nabla\cdot(\gamma-\wt\gamma)\nabla\wt u^{(2)}\quad \textrm{ in }\Omega.
\]
In a similar manner, we can derive the estimate for $u^{(2)}$ in~\eqref{eqn:est_u2}.
\end{proof} 
We are now ready to show the sensitivity result for uncertainty quantification. 
\begin{theorem}\label{THM:Uncertain Diff}
		Let $H_D$ be the internal data associated with both $(\Xi,\gamma,\sigma_a,\sigma_b)$ and $(\Xi,\wt\gamma, \wt\sigma_a,\wt\sigma_b)$. If we have that $\underline{g}:=\displaystyle\inf_{\partial\Omega} g>0$, then
\begin{equation}\label{EQ:Uncertainty Diff}
	\left\|\sigma_a-\wt\sigma_a\right\|_{L^p(\Omega)}+\left\|\sigma_b-\wt\sigma_b\right\|_{L^p(\Omega)} \le C  \left\|\frac{\wt\gamma-\gamma}{\wt\gamma}\right\|_{W^{1,p}(\Omega)},
	\end{equation}
    where $C$ is a positive constant depending on $\Omega,\gamma,g,\sigma_a$ and $\sigma_b$.  
\end{theorem}
\begin{proof}
	\textbf{(1). Estimate for $\sigma_a$.} Let $u$ and $\wt u$ be the solutions to the diffusion equation corresponding to $(\gamma, \sigma_a, \sigma_b)$ and $(\wt\gamma, \wt \sigma_a, \wt\sigma_b)$ respectively. Given that the corresponding data are the same, we have
	\begin{equation}\label{EQ:Equal Data Diff}
		\sigma_a u^{(1)}= \wt\sigma_a \wt u^{(1)},
	\end{equation}
giving
	\begin{equation*}
		(\sigma_a-\wt\sigma_a)u^{(1)} = -\wt\sigma_a(u^{(1)}-\wt u^{(1)}).
	\end{equation*}
This leads to
	\begin{equation}\label{EQ:Bound Diff 2}
		\|(\sigma_a-\wt\sigma_a)u^{(1)} \|_{W^{2,p}(\Omega)}\le C\|u^{(1)}-\wt u^{(1)}\|_{W^{2,p}(\Omega)}\leq C\left\|\frac{\wt\gamma-\gamma}{\wt \gamma}\right\|_{W^{1,p}(\Omega)}
	\end{equation}
	for some constant $C>0$, by Lemma~\ref{lemma:diffusion u1}. 
This yields that 
    \begin{equation*} 
    \|\sigma_a-\wt\sigma_a\|_{L^{p}(\Omega)}\leq C\left\|\frac{\wt\gamma-\gamma}{\wt \gamma}\right\|_{W^{1,p}(\Omega)}
    \end{equation*}	
    since $u^{(1)}$ is positive and bounded away from zero provided that $\underline{g}:=\displaystyle\inf_{\partial\Omega} g>0$.

\noindent\textbf{(2). Estimate for $\sigma_b$.}
Similarly, given the same data,
\begin{equation*}
\sigma_a u^{(2)}+2\sigma_b u^{(1)}u^{(1)}= \wt\sigma_a \wt u^{(2)} + 2\wt\sigma_b  \wt u^{(1)} \wt u^{(1)}.
\end{equation*}
This gives
\begin{equation*}
    2(\sigma_b-\wt\sigma_b)u^{(1)}u^{(1)}= -(\sigma_a-\wt\sigma_a)u^{(2)}-\wt\sigma_a(u^{(2)}-\wt u^{(2)}) - 2\wt\sigma_b \left[( u^{(1)} - \wt u^{(1)} )u^{(1)}+ \wt u^{(1)} (u^{(1)}-\wt u^{(1)})\right].
\end{equation*}
Due to elliptic regularity, the solution $u^{(2)}$ to \eqref{EQN:2nd linearization} satisfies $\|u^{(2)}\|_{W^{2,p}(\Omega)}\leq C \| g\|_{W^{2-1/p,p}(\Omega)}^2$.
Then Lemma~\ref{lemma:diffusion u1} yields that
\begin{equation*}\begin{split}
    \|(\sigma_b-\wt\sigma_b)u^{(1)}u^{(1)}\|_{L^p(\Omega)}&\leq  C\Big(\|u^{(2)}\|_{L^\infty(\Omega)}\|\sigma_a-\wt\sigma_a\|_{L^p(\Omega)}+\|u^{(2)}-\wt u^{(2)}\|_{L^p(\Omega)}\\
    &\quad +\left\| u^{(1)}-\wt u^{(1)}\right\|_{L^\infty(\Omega)}\left(\|u^{(1)}\|_{L^p(\Omega)}+\|\wt u^{(1)}\|_{L^p(\Omega)}\right)\Big)\\
    &\leq  C\Big(1+ \|g\|_{W^{2-1/p,p}(\Omega)} + \| g\|^2_{W^{2-1/p,p}(\Omega)} \Big)\left\|\frac{\wt\gamma-\gamma}{\wt \gamma}\right\|_{W^{1,p}(\Omega)}\\
    &\leq  C\left\|\frac{\wt\gamma-\gamma}{\wt \gamma}\right\|_{W^{1,p}(\Omega)}.
\end{split}\end{equation*}
We apply $u^{(1)}\ge \eps'>0$ for some $\eps'>0$ again. This proves \eqref{EQ:Uncertainty Diff}.

\end{proof}

\section{Concluding remarks}
\label{SEC:Concl}

In this work, we studied inverse coefficient problems for a semilinear radiative transport equation as well as its diffusion approximation. The aim was to reconstruct the first- and second-order absorption coefficients and the scattering coefficient from internal functionals of the coefficients and the solutions to the equations. The main applications we have in mind are those in quantitative photoacoustic imaging of optically heterogeneous media. Using the techniques of model linearization, we derived uniqueness as well as stability results on the reconstructions. In the transport regime, our results, based on the data encoded in the full albedo operator, supplement those in~\cite{ReZh-SIAM21} where uniqueness can only be derived for the reconstruction of the absorption coefficients, not the scattering coefficient, with finite number of internal data sets. In the diffusion regime, our result improved the linearized inversion of~\cite{ReZh-SIAM18}, again with more data.

There are many aspects of our results that can be improved. For instance, our results are obtained under the assumption that the boundary sources for the transport equation are small. This is far from what is required by real-world applications. It is assumed in practice that one has sufficiently strong sources to make the second-order effect in the transport equation (i.e. the quadratic term $\sigma_b \aver{u} u$) strong enough to be detected. Up to now, we do not even have a well-posedness theory, if it exists at all, for the semilinear transport equation with large boundary data. Moreover, our result requires data encoded in the full albedo operator (or generated from a $1$-parameter family of boundary sources in the diffusive regime) to reconstruct three unknown coefficients. It would be very interesting to see if it is possible to reconstruct the three coefficients with only three data sets (possibly generated from three specially selected boundary illuminations).

For applications in uncertainty quantification, we also derived the stability of reconstructing the absorption coefficients with respect to changes in the scattering coefficient; in Theorem~\ref{THM:Uncertain ERT} and Theorem~\ref{THM:Uncertain Diff} respectively. These results show that in the case that we do not have enough data to reconstruct all the coefficients, we can focus on the reconstruction of the absorption coefficients (which are often the mostly relevant ones in practical applications) while replacing the scattering coefficient with a good value from \emph{a priori} information. The error in the reconstruction in this case will not be too bad if the value of the scattering coefficient is not very different from its true value. Numerically uncertainty quantification, that is, evaluating the size of the constants in the stability bounds in~\eqref{EQ:Uncertainty ERT} and~\eqref{EQ:Uncertainty Diff}, following for instance the methods in~\cite{ReVa-SIAM20}, would be of great practical interests.

\section*{Acknowledgments}

This work is partially supported by the National Science Foundation through grants DMS-1937254, 
DMS-1913309 and DMS-2006731.

\appendix
\section{Appendix: The well-posedness result for the transport equation}
\label{SEC:Wellposed ERT}

Here we show the well-posedness of the semilinear transport equation~\eqref{EQ:ERT TP} for small boundary data. For simplicity, we use the notation
$$
\sigma(\bx):=\sigma_a(\bx)+\sigma_s(\bx) .
$$
We denote by $d_\Omega$ the diameter of the spatial domain $\Omega$, that is, 
$$
d_\Omega:={\rm diam}(\Omega).
$$
Based on the a-priori assumptions on $\sigma_a$ and $\sigma_s$, there is some constant $\nu>0$ so that
$$
0<\nu\leq {\sigma_a(\bx) \over \sigma_a(\bx)+\sigma_s(\bx) } <1 \quad \hbox{for all }\bx\in\Omega,
$$
which implies that
$$
{\sigma_s(\bx) \over \sigma_a(\bx)+\sigma_s(\bx) } \leq 1-\nu  \quad \hbox{for all } \bx\in\Omega.
$$

Let $L_{d\xi}^p(\Gamma_{-})$ be the usual space of $L^p$ functions on $\Gamma_-$ with measure $d\xi=|\bnu(\bx) \cdot \bv|d\mu(\bx) d\bv$, $d\mu(\bx)$ being the surface Lebesgue measure on $\partial\Omega$. Then we have the following result from~\cite{EgSc-AA14}.
\begin{proposition}\label{PRO:source L2} \cite[Theorem~1.2]{EgSc-AA14}
	Let $\Omega$ be bounded with Lipschitz boundary. Suppose that $ C_\infty = \|\sigma_s d_\Omega\|_{L^\infty(X)} <+\infty$. For any $g\in L_{d\xi}^{p}(\Gamma_{-})$ and $S\in L^p(X)$, $1\le p\le +\infty$, there exists a unique solution $u$ to the radiative transport equation
	\begin{equation}\label{EQ:u source}
	\begin{array}{rcll}
	\bv \cdot \nabla u(\bx, \bv)  + \sigma_a(\bx)  u(\bx, \bv) &=& \sigma_s K (u) +S(\bx,\bv), &\text{in}\ X\, \\
	u(\bx, \bv) &= & g(\bx, \bv), &\text{on}\  \Gamma_{-}\,
	\end{array}
	\end{equation}
	and $u$ satisfies 
	$$
	\|u\|_{L^p(X)}\leq  C_2\|S\|_{L^p(X)} + \wt{c}\|g\|_{L_{d\xi}^p(\Gamma_-)},
	$$
	where $C_2:={1\over \nu c_0}$, and $\wt c:=  {1\over \sqrt{\nu c_0}}$ when $p=2$ and $\wt{c} = 1$ when $p=\infty$. Here $c_0$ is defined in~\eqref{EQ:Transport Coeff}.
\end{proposition}

We need the following result on the existence of positive solutions for~\eqref{EQ:u source} when $S(\bx, \bv)\equiv 0$.
\begin{proposition}\label{PRO:positive} 
	Let $S(\bx, \bv)\equiv 0$ in~\eqref{EQ:u source}, and $g\in L_{d\xi}^{\infty}(\Gamma_{-})$ be given such that $\underline g:=\inf_{\Gamma_-} g>0$. Then, under the same assumptions in Proposition~\eqref{PRO:source L2}, then there exists an $\eps'>0$ such that the solution $u$ to~\eqref{EQ:u source} satisfies
	\[
		u(\bx, \bv) \ge \eps'>0, \ \ \mbox{in}\ \ X.
	\]
\end{proposition}
\begin{proof}
Proposition~\ref{PRO:source L2} ensures that there exists a unique solution $u$ satisfying
\[
	\|u\|_{L^\infty(X)}\le \|g\|_{L_{d\xi}^\infty(\Gamma_-)}.
\]
From standard transport theory~\cite{DaLi-Book93-6}, we know also that $u\ge 0$. Let us re-write~\eqref{EQ:u source}, with $S\equiv 0$, into the form
\begin{equation*}
	\begin{array}{rcll}
	\bv \cdot \nabla u(\bx, \bv)  + (\sigma_a+\sigma_s)(\bx)  u(\bx, \bv) &=& \sigma_s \dint_{\bbS^{d-1}} \Theta(\bv, \bv') u(\bx, \bv') d\bv', &\text{in}\ X\, \\
	u(\bx, \bv) &= & g(\bx, \bv), &\text{on}\  \Gamma_{-}\,.
	\end{array}
\end{equation*}
We can then integrate the equation by the method of characteristics to obtain that
\begin{align*}
	u(\bx,\bv)&= e^{-\int^{\tau_-(\bx,\bv)}_0 \sigma  (\bx-\eta\bv)d\eta} g(\bx-\tau_-(\bx,\bv)\bv,\bv)\\
	&\quad + \int^{\tau_-(\bx,\bv)}_0 \sigma_s(\bx-s\bv) e^{-\int^s_0\sigma (\bx-\eta\bv)d\eta} \int_{\mathbb{S}^{d-1}}  \Theta(\bv,\bv') u(\bx-s\bv,\bv')d\bv' ds
\end{align*} 
where $\sigma:=\sigma_a+\sigma_s$. Using $u\ge 0$ and $\Theta\geq 0$, we conclude that the second term is nonnegative. Therefore
\begin{align*}
	u(\bx,\bv)\ge e^{-\int^{\tau_-(\bx,\bv)}_0 \sigma  (\bx-\eta\bv)d\eta} g(\bx-\tau_-(\bx,\bv)\bv,\bv) \ge \underline g\, e^{-d_\Omega\overline\sigma }
\end{align*}
where $\overline\sigma:=\sup_{\Omega}\sigma$. The proof is complete if we define $\eps' :=\underline g\, e^{-d_\Omega  \overline \sigma}$. 
\end{proof} 


We have the following well-posedness result for~\eqref{EQ:ERT TP} with small data.
\begin{theorem}\label{THM:WELL}
	Let $\Omega\subset\bbR^d$ ($d\ge 2$) be an open convex bounded domain. Suppose that $\sigma_a,\sigma_b,\sigma_s$ satisfy~\eqref{EQ:Transport Coeff}. Then there exists a small parameter $0<\varepsilon<1$ such that when 
	$$
	g\in \mathcal{X}_\varepsilon:=\{g\in L_{d\xi}^\infty(\Gamma_-):\, \|g\|_{L_{d\xi}^\infty(\Gamma_-)}\leq \varepsilon\},
	$$
	the problem \eqref{EQ:ERT TP} has a unique small solution $u\in L^\infty(X)$ satisfying
	$$
	\|u\|_{L^\infty(X)}\leq C \|g\|_{L_{d\xi}^\infty(\Gamma_-)},
	$$
	with the constant $C>0$ being independent of $u$ and $g$.
\end{theorem}
\begin{proof}
We first consider the linear equation
\begin{equation*}
	\begin{array}{rcll}
		\bv \cdot \nabla u_0 + \sigma_a u_0 &=& \sigma_s K (u_0), &\text{in}\ X\, \\
	u_0 &= & g, &\text{on}\ \Gamma_{-}\, .
	\end{array}
	\end{equation*}
	By Proposition~\ref{PRO:source L2} with $p=\infty$, there exists a unique solution $u_0$ that satisfies
	\begin{align}\label{estimate u0}
	\|u_0\|_{L^\infty(X)}\leq \|g\|_{L_{d\xi}^\infty(\Gamma_-)}.
	\end{align}

	Let us now consider $w:=u-u_0$. If such function $w$ exists, then $w$ satisfies the problem:
	\begin{equation}\label{EQ: linear remainder}
	\begin{array}{rcll}
	\bv \cdot \nabla w + \sigma_a w &=& \sigma_s K (w) - G(w), &\text{in}\ X\, \\
	w &= & 0, &\text{on}\ \Gamma_{-}\,
	\end{array}
	\end{equation}
	with
	$$
	G(w):= \sigma_b \aver{u_0+w} (u_0+w).
	$$
	The problem is now to show the unique existence of $w$ to \eqref{EQ: linear remainder}. To this end, we will construct a contraction map and then apply the Contraction Mapping Principle. We first introduce the set of functions:
	$$
	\mathcal{M}:=\{\phi\in L^\infty(X):\, \phi|_{\Gamma_-}=0,\, \|\phi\|_{L^\infty(X)}\leq \delta\},
	$$
	where parameter $\delta$ will be determined later.
	For $\phi\in L^\infty(X)$, the source term $G(\phi)$ is also in $L^\infty(X)$. Therefore, the problem
	\begin{equation}\label{EQ: linear remainder 1}
	\begin{array}{rcll}
	\bv \cdot \nabla \wt w + \sigma_a \wt w &=& \sigma_s K (\wt w) - G(\phi), &\text{in}\ X\, \\
	\wt w &= & 0, &\text{on}\  \Gamma_{-}\,
	\end{array}
	\end{equation}
	is uniquely solvable due to Proposition~\ref{PRO:source L2}. We can therefore define the solution operator
	$$
	\mathcal{T}^{-1}: G(\phi)\in L^\infty(X)\mapsto \wt w\in L^\infty(X)
	$$
	to \eqref{EQ: linear remainder 1}. Moreover, by Proposition~\ref{PRO:source L2} again, we have
	\begin{align}\label{solution operator}
	\|\mathcal{T}^{-1}(G(\phi)) \|_{L^\infty(X)} \leq C_2 \|G(\phi)\|_{L^\infty(X)}.
	\end{align}

	Let us define the operator $F$ by 
	$$
		F(\phi):=\mathcal{T}^{-1}(G(\phi)) 
	$$
	for any $\phi\in\mathcal{M}$. In what follows, we will show that $F$ is contractive on the set $\mathcal{M}$ for appropriate parameter $\delta$.
	
	In the first step, we show that $F(\mathcal{M})\subset \mathcal{M}$. In fact, for any $\phi\in \mathcal{M}$, we have, by \eqref{solution operator}, that
	\begin{align}\label{contraction 1}
	\|F(\phi)\|_{L^\infty(X)} 
	&\leq C_2\|G(\phi)\|_{L^\infty(X)} \notag\\
	&\leq C_2\|\sigma_b \aver{u_0+\phi} (u_0+\phi)\|_{L^\infty(X)}\notag\\
	&\leq C_2C_0 (\varepsilon+\delta)^2
	\end{align}
	where $C_0$ is the constant introduced in~\eqref{EQ:Transport Coeff}. We can then take $\varepsilon, \delta$ sufficiently small so that $C_2C_0 (\varepsilon+\delta)^2<\delta$. This yields $F(\phi)\in \mathcal{M}$.
	
	In the second step, we show that $F$ is contractive on $\mathcal{M}$, that is, $\|F(\phi_1)-F(\phi_2)\|_{L^\infty(X)}<\|\phi_1-\phi_2\|_{L^\infty(X)}$ for any $\phi_1,\phi_2\in \mathcal{M}$. This follows from the following calculation:
	\begin{align*}
	\|F(\phi_1)-F(\phi_2)\|_{L^\infty(X)}
	&\leq C_2 \|G(\phi_1)-G(\phi_2)\|_{L^\infty(X)}\\
	&\leq C_2 \|\sigma_b \aver{u_0+\phi_1} (u_0+\phi_1)-\sigma_b \aver{u_0+\phi_2} (u_0+\phi_2)\|_{L^\infty(X)}\\
	&\leq C_2C_0 (\varepsilon+\delta)\|\phi_1-\phi_2\|_{L^\infty(X)}.
	\end{align*}
	By taking $\varepsilon$ and $\delta$ small enough, we can make $C_2C_0(\varepsilon+\delta)<1$. In this case, $F$ is contractive on $\mathcal{M}$.
	
	By applying the Contraction Mapping Principle, there exists a fixed point $w$ in $\mathcal{M}$ so that $F(w)=w$. Then $w$ is the solution to \eqref{EQ: linear remainder} and satisfies
	$$
	\|w\|_{L^\infty(X)}\leq C_2(\varepsilon+\delta)(\|u_0\|_{L^\infty(X)}+ \|w\|_{L^\infty(X)}) 
	$$
	due to \eqref{contraction 1}. By taking $\varepsilon, \delta$ even smaller if needed,  
	we have $C_2(\varepsilon+\delta)\|w\|_{L^\infty(X)}$ can be absorbed into the left-hand side, and thus,
	$$
	\|w\|_{L^\infty(X)}\leq C\|u_0\|_{L^\infty(X)}.
	$$
	We then conclude that $u=u_0+w$ is the solution to \eqref{EQ:ERT TP} and in particular, 
	$$
	\|u\|_{L^\infty(X)}\leq \|u_0\|_{L^\infty(X)}+\|w\|_{L^\infty(X)}\leq C \|u_0\|_{L^\infty(X)} \leq C \|g\|_{L_{d\xi}^\infty(\Gamma_-)}
	$$
	by combining \eqref{estimate u0} and the estimate above.
\end{proof}

In the following, we discuss briefly the differentiability of the solution.
For a nonzero $g \in L_{d\xi}^\infty(\Gamma_-)$ and small enough $\varepsilon_0>0$, let $u_\varepsilon = u(\bx,\bv;\varepsilon)$ be the solution to the problem \eqref{EQ:Semil ERT} with boundary data $\varepsilon g\in \mathcal{X}_{\varepsilon_0}$. 

We define the $k$-th derivative of $u_\varepsilon$ with respect to (w.r.t.) $\varepsilon$ by $u^{(k)}_\varepsilon: = \p^k_\varepsilon u_{\varepsilon}(\bx,\bv;\varepsilon)$ for $k=1,2$. In particular, the $k$-th derivative of $u_\varepsilon$ at $\varepsilon=0$ is denoted by $u^{(k)}$, instead of $u^{(k)}_\varepsilon|_{\varepsilon=0}=\p^k_\varepsilon u_\varepsilon |_{\varepsilon=0}$ for simplicity.  
We also define the linear operator $L$ by
$$
    L u := \bv \cdot \nabla u+ \sigma_a(\bx)u - \sigma_s(\bx) K(u).
$$

\begin{proposition}\label{Pro:differentiability} 
For $\varepsilon$ sufficiently small, $u^{(1)}_\varepsilon$ exists and satisfies
\begin{equation}\label{EQN:u1 epsilon}
	\begin{array}{rcll}
	L u^{(1)}_\varepsilon(\bx, \bv) + \sigma_b\aver{u_\varepsilon}  u^{(1)}_\varepsilon (\bx, \bv) +\sigma_b\aver{u^{(1)}_\varepsilon} u_\varepsilon(\bx, \bv) &=& 0 , &\text{in}\ X \\
		u^{(1)}_\varepsilon(\bx, \bv) &= & g(\bx, \bv), &\text{on}\ \Gamma_{-}.
	\end{array}
\end{equation}
Moreover, $u^{(2)}_\varepsilon$ exists and satisfies
\begin{equation}\label{EQN:u2 epsilon}
	\begin{array}{rcll}
		L u^{(2)}_\varepsilon(\bx, \bv) + \sigma_b\aver{u_\varepsilon}  u^{(2)}_\varepsilon (\bx, \bv) +\sigma_b\aver{u^{(2)}_\varepsilon}  u_\varepsilon(\bx, \bv) &=& -2\sigma_b \aver{u^{(1)}_\varepsilon}  u^{(1)}_\varepsilon(\bx,\bv) , &\text{in}\ X \\
		u^{(2)}_\varepsilon(\bx, \bv) &= & 0, &\text{on}\ \Gamma_{-}.
	\end{array}
\end{equation}
In particular, $u^{(1)}$ satisfies \eqref{EQ:u_1 0} and $u^{(2)}$ satisfies \eqref{EQ:ERT TP eps2}.
\end{proposition}
\begin{proof}
    Let $\Delta \varepsilon\neq 0$ and let $\tilde{u} = {u_{\varepsilon+\Delta\varepsilon} - u_\varepsilon \over \Delta \varepsilon}$, where $u_{\varepsilon+\Delta\varepsilon},u_\varepsilon\in\mathcal{X}_{\varepsilon_0}$. Then $\tilde{u}$ satisfies the linear transport equation with zero source
    \begin{equation*} 
    	\begin{array}{rcll}
    		L \tilde{u} (\bx, \bv) + \sigma_b\aver{\tilde{u}} u_{\varepsilon+\Delta\varepsilon}(\bx, \bv) +\sigma_b \aver{u_\varepsilon} \tilde{u}(\bx, \bv)   &=& 0 , &\text{in}\ X \\
    		\tilde{u}(\bx, \bv)  &= & g(\bx, \bv), &\text{on}\ \Gamma_{-}.
    	\end{array}
    \end{equation*}
Thus from Proposition~\ref{PRO:source L2}, we have
$$
    \|\tilde{u}\|_{L^\infty(X)}\leq C  \|g\|_{L_{d\xi}^\infty(\Gamma_-)},
$$
which yields that
$$
    \|u_{\varepsilon+\Delta\varepsilon} - u_\varepsilon \|_{L^\infty(X)}\leq C|\Delta \varepsilon| \|g\|_{L_{d\xi}^\infty(\Gamma_-)}.
$$
Let $w=\tilde{u}-v$, where $v$ is the solution to \eqref{EQN:u1 epsilon}. Then $w$ satisfies 
\begin{equation*} 
	\begin{array}{rcll}
		L w(\bx, \bv) + \sigma_b \aver{u_\varepsilon} w(\bx, \bv) +\sigma_b\aver{w}  u_\varepsilon(\bx, \bv) &=& -\sigma_b\aver{\tilde{u}}  (u_{\varepsilon+\Delta\varepsilon}-u_\varepsilon)(\bx, \bv) , &\text{in}\ X \\
		w (\bx, \bv)&= & 0, &\text{on}\ \Gamma_{-} 
	\end{array}
\end{equation*}
and also
$$
    \|w\|_{L^\infty(X)}\leq C  \|\sigma_b(u_{\varepsilon+\Delta\varepsilon}-u_\varepsilon)\aver{\tilde{u}}\|_{L^\infty(X)}\leq C |\Delta \varepsilon| \|g\|^2_{L_{d\xi}^\infty(\Gamma_-)}.
$$
Therefore when $\Delta \varepsilon\rightarrow 0$, $\tilde{u}$ converges to $v$ in $L^\infty(X)$, which implies that $u_\varepsilon$ is differentiable w.r.t. $\varepsilon$ and thus $u^{(1)}_\varepsilon=v$ exists.   
 
Let $\tilde{u}^{(1)} = {u^{(1)}_{\varepsilon+\Delta\varepsilon} - u^{(1)}_\varepsilon \over \Delta \varepsilon}$. Following a similar argument as above, we can also derive that $\tilde{u}^{(1)}$ converges in $L^\infty(X)$ and thus $u^{(2)}_\varepsilon$ exists. This completes the proof.
\end{proof}


\section{Appendix: The well-posedness result for the diffusion equation}
\label{SEC:Wellposed Diff}

Here we establish the well-posedness result for the boundary value problem~\eqref{EQ:Diff TP} with small boundary data. Let $\Omega$ be a bounded domain in $\bbR^d$ with smooth boundary $\partial\Omega$. We have the following theorem.
\begin{theorem}\label{THM:Wellposed Diff}
Assume that $\gamma, \sigma_a, \sigma_b$ satisfy \eqref{EQ:Diffusion Coeff}.
Let $p\in\bbR_+$ be such that $p>d$. Then there exists a small parameter $0<\varepsilon<1$ such that when  $f\in W^{2-1/p,p}(\partial\Omega)$ satisfies $\|f\|_{W^{2-1/p,p}(\partial\Omega)}\leq\varepsilon$,
the problem \eqref{EQ:Diff TP} has a unique solution $u\in W^{2,p}(\Omega)$ satisfying
\[
	\|u\|_{W^{2,p}(\Omega)}\leq c\|f\|_{W^{2-1/p,p}(\partial\Omega)}
\]
for some constant $c>0$ independent of $u$ and $f$. 
\end{theorem}
\begin{proof}
Similar as the transport case above, we apply the standard Contraction Mapping Theorem. We define the set of functions:
\[\mathcal{M}_D:=\{v\in W^{2,p}(\Omega)~|~ v|_{\p\Omega}=0 ,\,\|v\|_{W^{2,p}(\Omega)} \leq \delta\},\]
where $\delta>0$ will be determined later.
By standard theory on the well-posedness of linear elliptic equations (see for instance~\cite[Theorem 9.15 and Lemma 9.17]{GiTr-Book00}), we have that for $S(\bx) \in L^p(\Omega)$, the second-order equation
\begin{equation}\label{eqn:linear_diffusion}
\begin{array}{rcll}
-\nabla\cdot\gamma\nabla v+\sigma_a v&=&S, & \mbox{in}\ \ \Omega\\
v&=&f, & \mbox{on}\ \partial\Omega
\end{array}
\end{equation}
admits a unique solution $v\in W^{2,p}(\Omega)$ satisfying 
\[
\|v\|_{W^{2,p}(\Omega)}\leq c(\|f\|_{W^{2-1/p,p}(\partial\Omega)}+\|S\|_{L^p(\Omega)})
\]
for some constant $c>0$. Let $u_0$ be the solution to~\eqref{eqn:linear_diffusion} with $S=0$ and set $w=u-u_0$. Then we are set to find $w\in \mathcal{M}_D$ for $\delta>0$ small enough such that 
\[
\begin{array}{rcll}
-\nabla\cdot\gamma\nabla w+\sigma_a w&=&G_D(w), & \mbox{in}\ \ \Omega\\
w & = & 0, & \mbox{on}\ \partial\Omega
\end{array}
\]
where  
$$
    G_D(w):=-\sigma_b {(u_0+w)^2}.  
$$
This is equivalent to find a fixed point in $\mathcal{M}_D$ to the contractive operator $F_D:=\mathcal{T}_D^{-1}\circ G_D$, where $\mathcal{T}_D^{-1}: L^p(\Omega)\rightarrow W^{2,p}(\Omega)$ denotes the bounded operator $S\mapsto u_S$ with $u_S$ being the solution of \eqref{eqn:linear_diffusion} with $f=0$. 

We first show that $F_D(\mathcal{M}_D)\subset \mathcal{M}_D$. By the Sobolev Embedding Theorem, when $p>d$, we have $W^{2,p}(\Omega)\hookrightarrow \cC^{1,1-d/p}(\overline\Omega)$. Therefore, for $\phi\in \mathcal{M}_D$, we obtain
\begin{equation}\label{eqn:F_est}
	\begin{split}
		\|G_D(\phi)\|_{L^p(\Omega)}&\leq c \|u_0+\phi\|_{L^\infty(\Omega)}\|u_0+\phi\|_{L^p(\Omega)}\\
		&\leq c\|u_0+\phi\|_{W^{2,p}(\Omega)}^2\\
		&\leq c(\|f\|_{W^{2-1/p,p}(\p\Omega)}^2+\|\phi\|_{W^{2,p}(\Omega)}^2) \leq c(\varepsilon^2+\delta^2).
	\end{split}
\end{equation}
Hence, 
\[
	\|F_D(\phi)\|_{W^{2,p}(\Omega)}\leq c\|G_D(\phi)\|_{L^p(\Omega)}\leq c(\varepsilon^2+\delta^2)<\delta
\]
when $\delta>0$ and $\delta>\varepsilon>0$ are small enough. This then leads to $F_D(\phi)\in\mathcal{M}_D$.

Next we show that the map $F_D$ is contractive on $\mathcal{M}_D$. Take any $\phi_1,\phi_2\in \mathcal{M}_D$, we have
\[
\begin{split}
	\|F_D(\phi_1)-F_D(\phi_2)\|_{W^{2,p}(\Omega)}&\leq c\|G_D(\phi_1)-G_D(\phi_2)\|_{L^p(\Omega)}.
\end{split}
\]
Using the fact that 
\[
\begin{split}
|G_D(\phi_1)-G_D(\phi_2)|&=|\sigma_b| { |(u_0+\phi_1)^2 - (u_0+\phi_2)^2 | }\\
&\leq c|\phi_1-\phi_2|(2|u_0|+|\phi_1|+|\phi_2|),
\end{split}
\]
we obtain
\[
\begin{split}
	\|G_D(\phi_1)-G_D(\phi_2)\|_{L^p(\Omega)}&\leq c\|\phi_1-\phi_2\|_{L^p(\Omega)}\left(2\|u_0\|_{L^\infty(\Omega)}+\|\phi_1\|_{L^\infty(\Omega)}+\|\phi_2\|_{L^\infty(\Omega)}\right)\\
&\leq c\|\phi_1-\phi_2\|_{W^{2,p}(\Omega)}\left(2\|u_0\|_{W^{2,p}(\Omega)}+\|\phi_1\|_{W^{2,p}(\Omega)}+\|\phi_2\|_{W^{2,p}(\Omega)}\right)\\
&\leq c(\varepsilon+\delta)\|\phi_1-\phi_2\|_{W^{2,p}(\Omega)},
\end{split}
\]
which leads to 
$$
\|F_D(\phi_1)-F_D(\phi_2)\|_{W^{2,p}(\Omega)}\leq c(\varepsilon+\delta)\|\phi_1-\phi_2\|_{W^{2,p}(\Omega)}.
$$
This implies that $F_D$ is a contraction on $\mathcal{M}_D$ when $\delta, \varepsilon$ are sufficiently small. The Contraction Mapping Theorem then concludes that $F_D$ has a unique fixed point $v\in \mathcal{M}_D$ such that $F_D(v)=v$. Therefore $u=u_0+v$ is the solution to \eqref{EQ:Diff TP}. Moreover, following a similar argument as in \eqref{eqn:F_est}, we can derive
\[
	\|v\|_{W^{2,p}(\Omega)}=\|F_D(v)\|_{W^{2,p}(\Omega)}\leq c\|G_D(v)\|_{L^p(\Omega)}\leq c(\varepsilon +\delta) (\|u_0\|_{W^{2,p}(\Omega)}+ \|v\|_{W^{2,p}(\Omega)}).
\]
Choosing $\delta>\varepsilon>0$ small enough, the term containing $\|v\|_{W^{2,p}(\Omega)}$ on the right-hand side of the above estimate can be absorbed by the left-hand side. Therefore, this implies
\[
	\|v\|_{W^{2,p}(\Omega)}\leq c\|u_0\|_{W^{2,p}(\Omega)}.
\]
Finally we have
\[
	\|u\|_{W^{2,p}(\Omega)}=\|u_0+v\|_{W^{2,p}(\Omega)}\leq c\|u_0\|_{W^{2,p}(\Omega)}\leq c\|f\|_{W^{2-1/p,p}(\partial\Omega)}.
\]
The proof is complete.
\end{proof}


 
 


\end{document}